\documentclass{article}
\usepackage[utf8]{inputenc}
\usepackage{a4wide}
\usepackage[UKenglish]{babel}
\usepackage{graphicx}
\usepackage{subfigure}
	\graphicspath{{./}{figures/}}
\usepackage[colorlinks=true,linkcolor=blue,citecolor=red]{hyperref}
\usepackage{amssymb,empheq,bbold}
\usepackage[inline]{enumitem}
\usepackage{geometry}
\usepackage{caption}
\usepackage{amsthm}
\usepackage{bbm}
\usepackage[title]{appendix}
\usepackage{amsmath,amsthm,bm}
\usepackage{amsfonts}
\usepackage{algorithmic}
\usepackage{algorithm}
\usepackage{float}
\usepackage{overpic}
\usepackage{booktabs}
\usepackage{multirow}

\newcommand \commentout[1] {}

\newcommand{\R}{\mathbb{R}}

\newtheorem{theorem}{Theorem}

\newtheorem{proposition}[theorem]{Proposition}
\newtheorem{lemma}[theorem]{Lemma}
\theoremstyle{definition}
\newtheorem{remark}{Remark}

\title{An operator-splitting algorithm for the hypergraph $p$-Laplacian with applications to missing data recovery}

\author{
Kehan Shi\thanks{Department of Mathematics, China Jiliang University, Hangzhou 310018 China (kshi@cjlu.edu.cn).}
\and
Jin Liu\thanks{Helmholtz Imaging, Deutsches Elektronen-Synchroton DESY, Notkestr. 85,
22607 Hamburg, Germany (jin.liu1@desy.de).}
\and
 Martin Burger$^\dagger$\thanks{Department of Mathematics, University of Hamburg, Bundesstr. 55, 20146
Hamburg, Germany (martin.burger@desy.de).}
}
\date{}

\graphicspath{{figures/}}

\begin{document}

\maketitle

\begin{abstract}
Hypergraph $p$-Laplacian regularization is a fundamental model in data analysis with successful applications in various tasks. It aims to minimize a nonsmooth and typically large-scale objective function defined as the sum of the $p$-th powers of the Lipschitz regularization over hyperedges.
In this paper, we propose an operator-splitting algorithm for the hypergraph $p$-Laplacian that allows us to handle hyperedges separately in a Gauss-Seidel fashion. Each subproblem can be viewed as a generalized graph Lipschitz learning on a hyperedge, for which we introduce an auxiliary variable to overcome the nonsmoothness and solve it with one step of the alternating direction method of multipliers (ADMM).
The resulting algorithm performs proximal ADMM updates sequentially over the hyperedges, and its convergence is proven.
We test the algorithm on missing data recovery problems, including image sparse inpainting and semi-supervised learning, to demonstrate that it is faster than existing methods.
\end{abstract}

\medskip


\medskip


\section{Introduction}
Let $H=(V, E_H)$ be an undirected hypergraph with vertex set $V=\{x_i\}_{i=1}^n$ and hyperedge set $E_H=\{e_k\}_{k=1}^m$.
Each hyperedge $e_k$ is a subset of $V$ with cardinality $|e_k|\geq 2$.
For any vertices $x_i, x_j\in V$, we denote by $w_{i,j}\geq 0$ the weight between $x_i$ and $x_j$.
In this paper, we are interested in the following optimization problem
\begin{equation}\label{eq:1.1}
  \min_{u\in\mathbb{R}^n}
   J(u)=J_H(u)+ F(u),
\end{equation}
where
\begin{equation}\label{eq:1.1a}
J_H(u)=\sum_{k=1}^mJ_{e_k}(u)=\sum_{k=1}^m \max_{x_i,x_j\in e_k}w_{i,j}|u(x_i)-u(x_j)|^p, \quad p\geq 1,
\end{equation}
denotes the regularization of $u$ on the hypergraph $H$ and $F(u)$ is a data fidelity term.

The regularization $J_H$ covers three classical graph and hypergraph cases.
If the hypergraph $H$ degenerates to a graph $G$, i.e., the cardinality $|e_k|=2$ for all $1\leq k\leq m$, then $J_H$ becomes the graph $p$-Laplacian \cite{elmoataz2015p,flores2022analysis}
\begin{equation}\label{eq:GL}
  J_G(u)=\sum_{i,j=1}^n w_{i,j}|u(x_i)-u(x_j)|^p,
\end{equation}
where we adopt the convention that $w_{i,j}=0$ if there is no edge between $x_i$ and $x_j$.
When the hypergraph $H$ contains only one hyperedge, i.e., $m = 1$, replacing $w_{i,j}$ with $w_{i,j}^p$ in \eqref{eq:1.1a} makes the exponent $p$ irrelevant, and $J_H$ simplifies to the graph Lipschitz regularization \cite{kyng2015algorithms,roith2023continuum}
\begin{equation}\label{eq:1.2}
  J_L(u)=\max_{x_i,x_j\in V}w_{i,j}|u(x_i)-u(x_j)|.
\end{equation}
Finally, if the weights are identical on each hyperedge, i.e., $w_{i,j}=w_k$ for any $x_i,x_j\in e_k$, then $J_H$ reduces to the standard hypergraph $p$-Laplacian \cite{hein2013total}
\begin{equation*}
  \bar{J}_H(u)=\sum_{k=1}^m w_k\max_{x_i,x_j\in e_k}|u(x_i)-u(x_j)|^p.
\end{equation*}
In the case $p = 1$, $\bar{J}_H$ is known as the hypergraph total variation.
Making use of pairwise weights $w_{i,j}$ in \eqref{eq:1.1a} is natural and advantageous, particularly in scenarios such as point cloud data processing, where pairwise weights are available \cite{shi2025hypergraph}.

The data fidelity term $F$ in $\eqref{eq:1.1}$ is tailored to the specific task.
In this paper, we consider the application of the hypergraph $p$-Laplacian \eqref{eq:1.1}--\eqref{eq:1.1a} to missing data recovery.
Let
$F(u)=\mathbb{I}_{L}(u)$, where
\begin{align}\label{eq:1.4}
   \mathbb{I}_{L}(u)=
   \begin{cases}
     0,\quad &\mbox{if } u(x_i)=y_i \mbox{ for all } x_i\in L,\\
     \infty,\quad &\mbox{otherwise},
   \end{cases}
 \end{align}
$L$ is a subset of $V$, and $y_i\in\mathbb{R}$.
Problem \eqref{eq:1.1}--\eqref{eq:1.1a} becomes a data interpolation model that recovers the missing values of $u$ on $V\backslash L$. The given data $\{(x_i,y_i), x_i\in L, y_i\in \mathbb{R}\}$ are commonly referred to as the training set. Its applications include image sparse inpainting, semi-supervised learning, and related tasks.
Instead of the hard constraint \eqref{eq:1.4}, one may also consider the soft constraint
\begin{equation*}
  F(u)=\sum_{x_i\in L}|u(x_i)-y_i|^2,
\end{equation*}
which is more robust to noise in the data.

This paper focuses on the numerical algorithm for the optimization problem \eqref{eq:1.1}--\eqref{eq:1.1a}, which is challenging due to the non-differentiability and large-scale nature of $J_H$.
Note that $F_H$ is convex. Problem \eqref{eq:1.1}--\eqref{eq:1.1a} can be solved by the subgradient descent method \cite{shor2012minimization}.
However, it requires a careful selection of the step size \cite{zhang2017re}.
In \cite{hein2013total}, the authors proposed to use the decomposition
\begin{equation*}
  \max_{x_i,x_j\in e_k}|u(x_i)-u(x_j)|= \max_{x_i\in e_k} u(x_i)-\min_{x_i\in e_k} u(x_i)
  =\max K_{e_k} u-\min K_{e_k} u,
\end{equation*}
where $K_{e_k} \in \mathbb{R}^{|e_k| \times n}$ has rows corresponding to the standard unit vectors for $x_i \in e_k$.
Then $\max K_{e_k} u$ and $\min K_{e_k} u$ can be handled separately and efficiently by the primal-dual hybrid gradient (PDHG) algorithm \cite{chambolle2011first}.
This strategy works for the objective function $\bar{J}_H$ with $p = 1$ and $p = 2$, but does not extend to general $p$ or to the objective function $J_H$.

For large-scale optimization problems, stochastic optimization methods have attracted considerable attention.
In \cite{shi2025hypergraph}, the authors applied the stochastic PDHG algorithm \cite{chambolle2018stochastic} to problem \eqref{eq:1.1}--\eqref{eq:1.1a} by rewriting $J_H$ as
 \begin{equation*}
   J_H(u)
   =\sum_{k=1}^m\max_{x_i,x_j\in e_k}\left|w_{i,j}^{\frac{1}{p}}(u(x_i)-u(x_j))\right|^p
   =\sum_{k=1}^m g(D_k u),
 \end{equation*}
where
\begin{equation*}
  g(\beta)=\|\beta\|_\infty^p,\quad
  N=\max\left\{{|e_k|(|e_k|-1)}/{2}: k=1,\cdots,m\right\}, \quad
  D_k\in\mathbb{R}^{N\times n},
\end{equation*}
such that
\[
D_ku=w_{i,j}^{\frac{1}{p}}(u(x_i)-u(x_j)).
\]
The term “stochastic” refers to the fact that, instead of updating all $m$ dual variables simultaneously, only one dual variable is selected randomly for update at each iteration.
This approach works well when $|e_k|$ and $m$ are of moderate size.
If $|e_k|$ is large, the dimension of $D_k$ grows substantially, resulting in considerable memory consumption and computational cost.
Since each update of a dual variable requires a corresponding update of the primal variable $u \in \mathbb{R}^n$, a large $m$ causes frequent updates of $u$, which can be computationally expensive for large $n$.

In this paper, we propose a new algorithm based on the operator-splitting method \cite{glowinski2016some}, also known as the incremental proximal point method \cite{bertsekas2011incremental}, for problem \eqref{eq:1.1}--\eqref{eq:1.1a}.
It allows us to handle each $J_{e_k}(u)$ for $k = 1, \dots, m$ separately.
The numerical algorithm reads
\begin{equation}\label{eq:1.3}
  u^{k+1}=\arg\min_{u\in\mathbb{R}^n}\max_{x_i,x_j\in e_{i_k}}w_{i,j}|u(x_i)-u(x_j)|^p
  +\frac{1}{m}F(u)
  +\frac{1}{2\tau}\|u-u^k\|^2,
\end{equation}
where $i_k=\mbox{mod}(k,m)+1$ and $\tau>0$ is the step size.
As the number of iterations $k$ increases, the algorithm processes all hyperedges in cyclic order.
If $F(u)=\mathbb{I}_{L}(u)$, $u^{k+1}(x_i)=u^{k}(x_i)$ for $x_i\in V\backslash(e_{i_k}\cup L)$ and $u^{k+1}(x_i)=y_i$ for $x_i\in L$.
Subproblem \eqref{eq:1.3} aims to solve $u$ on $e_{i_k}$, denoted by $u|_{e_{i_k}}$, which is a low-dimensional optimization problem in $\mathbb{R}^{|e_{i_k}|}$.
Given that $|e_{i_k}|\ll n$ in most cases, this approach avoids updating the entire $u\in\mathbb{R}^n$.
The operator-splitting method is in the spirit of the Gauss–Seidel method, where the obtained $u|_{e_{i_k}}$
is used for solving $u|_{e_{i_{k+1}}}$.

Subproblem \eqref{eq:1.3} can be viewed as a generalized Lipschitz learning problem on the hyperedge $e_k$ with the data fidelity term $F$ and the soft constraint $|u - u^k|^2$.
The standard Lipschitz learning problem with a hard constraint,
 \begin{equation}\label{eq:1.5}
  \min_{u\in\mathbb{R}^n}\max_{x_i,x_j\in V} w_{i,j}|u(x_i)-u(x_j)|
  +\mathbb{I}_{L}(u)
\end{equation}
admits non-unique minimizers. One usually seeks the unique lex-minimizer that achieves a lexicographically minimal gradient.
This is equivalent to solving the graph $\infty$-Laplacian equation \cite{kyng2015algorithms,calder2019consistency}
 \begin{equation*}
 \begin{cases}
   \max_{x_j\in V}w_{i,j}(u(x_j)-u(x_i))+\min_{x_j\in V}w_{i,j}(u(x_j)-u(x_i))=0, & \mbox{ if } x_i\in V\backslash L, \\
   u(x_i)=y_i, & \mbox{ if } x_i\in L.
 \end{cases}
 \end{equation*}
This approach is no longer valid for subproblem \eqref{eq:1.3} due to the soft constraint and the $p$-th power with $p > 1$.

Instead, we introduce an auxiliary variable $d_{i_k}$ for \eqref{eq:1.3} and rewrite it as
\begin{equation*}
  u^{k+1}=\arg\min_{u\in\mathbb{R}^n}\|d_{i_k}\|_\infty ^p
  +\mathbb{I}_{L}(u)
  +\frac{1}{2\tau}\|u-u^k\|^2,\quad \mbox{s.t. } d_{i_k}=D_{i_k} u.
\end{equation*}
This problem can be solved by the standard alternating direction method of multipliers (ADMM) \cite{boyd2011distributed}.
The algorithm is of independent interest for Lipschitz learning \eqref{eq:1.5}, as it provides an alternative minimizer for the non-unique problem.
One may also employ PDHG \cite{chambolle2011first} for this problem.
However, the ADMM approach avoids the explicit use of the potentially large matrix $D_k$.
Moreover, the ADMM minimizer has a more intuitive interpretation, as its first iteration (with a zero initial guess) coincides with Laplace learning (see Remark \ref{remark:1}).

Combining the operator-splitting method with the ADMM solver for subproblem \eqref{eq:1.3} leads to a new algorithm for problem \eqref{eq:1.1}--\eqref{eq:1.1a}.
For computational efficiency, we restrict ourselves to a single ADMM iteration with an appropriate initialization for each solve of \eqref{eq:1.3}.
The proposed algorithm can be interpreted both as an inexact incremental proximal point method \cite{bertsekas2011incremental} and as a variant of incremental ADMM \cite{zhang2011unified}.
The particular structure of $J_H$ and the convexity of $F$ allow us to prove the convergence of the algorithm within the ADMM framework \cite{zhang2011unified}.
Notably, the convergence result does not rely on any assumption on the step size $\tau$, in contrast to the incremental proximal point method \cite{bertsekas2011incremental}.

This paper is organized as follows. In Section 2, we present the details of the algorithm.
The convergence of the algorithm is proven in Section 3.
Section 4 is devoted to numerical experiments, where we compare the proposed algorithm with related methods for missing data recovery tasks, including image sparse inpainting and semi-supervised learning.
It indicates that the proposed algorithm is several times faster than existing algorithms for large-scale datasets.
We conclude the paper in Section 5.

\section{Proposed algorithm}
Let $\langle \cdot, \cdot\rangle$, $\|\cdot\|$, and $\|\cdot\|_\infty$ be the inner product, the $\ell^2$ norm, and the $\ell^\infty$ norm in the Euclidean space, respectively.
For notational convenience, a vector $d\in\mathbb{R}^{n_1 n_2}$ is sometimes written in a matrix form $d\in \mathbb{R}^{n_1\times n_2}$. But keep in mind that the norm involving $d$ refers specifically to the vector norm in $\mathbb{R}^{n_1 n_2}$.

For any convex function $J:\mathbb{R}^n\rightarrow \mathbb{R}$, the subdifferential of $J$ at $u\in\mathbb{R}^n$ is denoted by $\partial J(u)$, i.e.,
\[
\partial J(u)= \left\{\, q \in \mathbb{R}^n \;\middle|\;
J(v) \geq J(u) + \langle q,\, v-u \rangle \;\; \text{for all } v \in \mathbb{R}^n \right\}.
\]
Throughout this paper, we always assume that the data fidelity term $F$ in \eqref{eq:1.1} is proper and convex.
Specifically, we shall consider $F(u)=\mathbb{I}_L(u)$ in numerical experiments.

\subsection{Operator-splitting algorithm}
Let us begin with the optimality condition for \eqref{eq:1.1}--\eqref{eq:1.1a}, which is
 \begin{equation*}
   \sum_{k=1}^m\partial J_{e_k}(u)+\partial F(u)\ni 0.
 \end{equation*}
 We associate with it the following initial value problem
 \begin{equation}\label{eq:3.4}
   \begin{cases}
   \frac{\partial u}{\partial t}+\sum_{k=1}^m\partial J_{e_k}(u)+\partial F(u)\ni 0,\quad t\in(0,\infty),\\
   u(0)=u^{0},
   \end{cases}
 \end{equation}
 where $u^{0}$ is an initial guess of the solution.
 Equation \eqref{eq:3.4} is now in a standard form that can be solved using the operator-splitting method.

 We adopt Lie's scheme \cite{glowinski2016some} for the time-discretization of \eqref{eq:3.4}.
 Let $\tau$ be the step size, $t^i=i\tau$, and $u^i$ be an approximation of $u(t^i)$.
 For any $i\geq 0$, we solve each subproblem in a cyclic order
 \begin{equation}\label{eq:3.5}
   \begin{cases}
   \frac{\partial u}{\partial t}+\partial J_{e_k}(u)+\frac{1}{m}\partial F(u)\ni 0,\quad t\in(t^i,t^{i+1}),\\
   u(t^i)=u^{i+(k-1)/m},
   \end{cases}
 \end{equation}
 and update
 \begin{equation*}
   u^{i+k/m}=u(t^{i+1}),
 \end{equation*}
 for $k = 1,2,\cdots, m$.

Subproblem \eqref{eq:3.5} has no closed-form solution, even in the case $F\equiv 0$.
 We further discretize it using one step of the backward Euler scheme, also known as the Marchuk-Yanenko scheme. Equation \eqref{eq:3.5} is then approximated by
 \begin{equation*}
   \frac{u-u^{i+(k-1)/m}}{\tau}+\partial J_{e_k}(u)
   +\frac{1}{m}\partial F(u)\ni 0.
 \end{equation*}
It corresponds to the optimality condition of the following minimization problem
\begin{equation*}
  \min_{u\in\mathbb{R}^n} J_{e_k}(u) + \frac{1}{m} F(u) + \frac{1}{2\tau}\left\|u- u^{i+(k-1)/m}\right\|^2.
\end{equation*}
Substituting it into \eqref{eq:3.5} yields the operator-splitting scheme
\begin{equation}\label{eq:3.7}
  u^{k+1}=\arg\min_{u\in\mathbb{R}^n}J_{e_{i_k}}(u)  +\frac{1}{m} F(u)+ \frac{1}{2\tau}\left\|u- u^{k}\right\|^2
\end{equation}
for problem \eqref{eq:1.1}--\eqref{eq:1.1a}, where $i_k=\mbox{mod}(k,m)+1$.

Scheme \eqref{eq:3.7} can also be motivated by the incremental (more precisely, cyclic) proximal point method for large-scale optimization problems \cite{bertsekas2011incremental}.
In the previous discussion, we considered the operator-splitting scheme as an approximation of the evolution equation \eqref{eq:3.4}.
This provides an intuitive interpretation of the initial guess $u^0$, which can be useful in applications.
Let $F(u) \equiv 0$.
The evolution of \eqref{eq:3.4} can be interpreted as performing hypergraph regularization on the initial guess $u^0$.
If $u^0$ represents an image, the iteration \eqref{eq:3.7} generates a sequence of progressively regularized images ${u^k}$ that become smoother as $k$ increases.
This procedure preserves the mean intensity of the image, in the same way as many PDE-based models for image processing \cite{aubert2006mathematical}.

\begin{proposition}
  Let $F(u)\equiv 0$ and the sequence $\{u^k\}$ be generated by \eqref{eq:3.7}. Then for any $k\geq 0$,
  \[
  \frac{1}{n}\sum_{i=1}^{n}u^k(x_i)=\frac{1}{n}\sum_{i=1}^{n}u^0(x_i).
  \]
\end{proposition}
\begin{proof}
  Let $u\in\mathbb{R}^n$ and  $q\in \partial J_{e_{i_k}}(u)$. It follows that
  \[
  J_{e_{i_k}}(v)-J_{e_{i_k}}(u)\geq \langle q, v-u\rangle, \quad \mbox{for any } v\in\mathbb{R}^n.
  \]
  By taking $v=u+1$ and $v=u-1$ in the above respectively, we have
  \[
  \langle q, 1\rangle =0.
  \]
  Namely, the subdifferential $J_{e_{i_k}}$ has the zero mean value at any $u\in\mathbb{R}^n$.

  Notice from \eqref{eq:3.7} that
  \[
  \frac{1}{\tau} (u^{k+1}-u^k) +\partial J_{e_{i_k}}(u^{k+1})\ni 0.
  \]
  By using the zero mean value of $\partial J_{e_{i_k}}(u^{k+1})$, we obtain the conclusion.
\end{proof}

Due to the non-differentiable maximum operator in $J_{e_{i_k}}$, each subproblem in \eqref{eq:3.7} must be solved using an additional numerical method.
We propose to handle it with the alternating direction method of multipliers (ADMM) \cite{boyd2011distributed}.
To this end, we introduce an auxiliary variable $d_{i_k} \in \mathbb{R}^{n \times n}$ and adopt the following notation
  \begin{align}\label{eq:Du}
    (D_{i_k}u)(x_i,x_j)=
    \begin{cases}
      w^{\frac{1}{p}}_{i,j}(u(x_i)-u(x_j)), & \mbox{if } x_i,x_j\in e_{i_k}, \\
      0, & \mbox{otherwise}.
    \end{cases}
  \end{align}
Then \eqref{eq:3.7} is equivalent to the constrained problem
\begin{equation*}
  \min_{u\in\mathbb{R}^n,d\in\mathbb{R}^{n\times n}} \|d_{i_k}\|^p_\infty  +\frac{1}{m} F(u)+ \frac{1}{2\tau}\left\|u- u^{k}\right\|^2
   \quad \mbox{ s.t. }  D_{i_k}u=d_{i_k}.
\end{equation*}
The associated augmented Lagrangian reads
\begin{equation*}
  \mathcal{L}_\rho(u,d_{i_k};\lambda_{i_k})
  =
  \|d_{i_k}\|^p_\infty+\frac{1}{m} F(u)+ \frac{1}{2\tau}\left\|u- u^{k}\right\|^2
  +\langle \lambda_{i_k}, D_{i_k}u-d_{i_k}\rangle+\frac{\rho}{2}\|D_{i_k}u-d_{i_k}\|^2,
\end{equation*}
where $\lambda_{i_k}\in\mathbb{R}^{n\times n}$ is the Lagrangian multiplier and $\rho>0$ is the penalty parameter.
Instead of jointly solving $u$ and $d_{i_k}$ in $\mathcal{L}_\rho$, ADMM alternatively updates $u$ and $d_{i_k}$.
Namely, for $j=0,1,\cdots$, and initializations $u^{k,0}\in\mathbb{R}^n$, $d_{i_k}^{0}, \lambda_{i_k}^{0}\in\mathbb{R}^{n\times n}$,
\begin{align}\label{eq:admm}
  \begin{cases}
    u^{k,j+1}=\arg\min_{u\in\mathbb{R}^n}
  \frac{1}{m} F(u)+\frac{1}{2\tau} \|u-u^k\|^2
  +\langle \lambda_{i_k}^{j}, D_{i_k}u-d_{i_k}^{j}\rangle+\frac{\rho}{2}\|D_{i_k}u-d_{i_k}^{j}\|^2, \\
   d_{i_k}^{j+1}= \arg\min_{d\in\mathbb{R}^{n\times n}} \|d\|^p_\infty
  +\langle \lambda_{i_k}^{j}, D_{i_k}u^{k,j+1}-d\rangle+\frac{\rho}{2}\|D_{i_k}u^{k,j+1}-d\|^2, \\
   \lambda_{i_k}^{j+1}= \lambda_{i_k}^{j}+\rho(D_{i_k}u^{k,j+1}-d_{i_k}^{j+1}).
  \end{cases}
\end{align}
The convergence of the algorithm is classical.
We refer the readers to \cite{zhang2011unified} for the proof.

Applying \eqref{eq:admm} to \eqref{eq:3.7} yields the final scheme for the hypergraph $p$-Laplacian \eqref{eq:1.1}--\eqref{eq:1.1a}.
For computational efficiency, we do not require an exact solution of subproblem \eqref{eq:3.7}. In fact, for each subproblem, only one iteration of the ADMM scheme \eqref{eq:admm} is performed, i.e., $j = 0$ in \eqref{eq:admm}. Moreover, a warm-start strategy is employed for the initialization of ADMM: the variables $(u, d_1, \dots, d_m, \lambda_1, \dots, \lambda_m)$ obtained from solving for $u^{k}$ are used as the initial values for computing $u^{k+1}$.
The details of the algorithm are summarized in Algorithm \ref{alg1}.

\begin{algorithm}[H]
  \caption{Operator-splitting algorithm for the hypergraph $p$-Laplacian \eqref{eq:1.1}--\eqref{eq:1.1a}.}
  \label{alg1}
  \begin{algorithmic}
  \REQUIRE Weight $W=(w_{i,j})\in\mathbb{R}^{n\times n}$, data fidelity $F$,
  parameters $p\geq 1$, $m\geq 1$, $\tau>0$, $\rho>0$.
  \STATE Initialization: $u^{0}\in\mathbb{R}^n$, $d_i^0=\lambda_i^0=0\in \mathbb{R}^{n\times n}$
  for $i=1,\cdots,m$, $k=0$.
  \WHILE { the stopping criterion is not satisfied}
  \STATE Let $i_k=\mbox{mod}(k,m)+1$.
  \STATE Update $u$:
  \begin{equation}\label{eq:3.8}
    u^{k+1}=\arg\min_{u\in\mathbb{R}^n}
  \frac{1}{m} F(u)+\frac{1}{2\tau} \|u-u^k\|^2+\langle \lambda_{i_k}^{k}, D_{i_k}u-d_{i_k}^{k}\rangle+\frac{\rho}{2}\|D_{i_k}u-d_{i_k}^{k}\|^2.
  \end{equation}
  \STATE Update $d$:
  \begin{align}\label{eq:3.9}
    d^{k+1}_l=
    \begin{cases}
    \arg\min\limits_{d\in\mathbb{R}^{n\times n}} \|d\|^p_\infty
  +\langle \lambda_{i_k}^{k}, D_{i_k}u^{k+1}-d\rangle+\frac{\rho}{2}\|D_{i_k}u^{k+1}-d\|^2,
    ~ &\mbox{if } l= i_k, \\
    d^{k}_l, ~ &\mbox{if } l\neq i_k.
  \end{cases}
  \end{align}
  \STATE Update $\lambda$:
  \begin{align}\label{eq:3.10}
  \lambda^{k+1}_l=
  \begin{cases}
    \lambda_{i_k}^{k}+\rho(D_{i_k}u^{k+1}-d_{i_k}^{k+1}),
    ~ &\mbox{if } l= i_k, \\
    \lambda^{k}_l, ~ &\mbox{if } l\neq i_k.
  \end{cases}
  \end{align}
  \STATE Update $k=k+1$.
  \ENDWHILE
  \RETURN $u^{k}$.
  \end{algorithmic}
  \end{algorithm}

For notational simplicity, we assume that $d_l^k, \lambda_l^k\in\mathbb{R}^{n\times n}$ in Algorithm \ref{alg1}. In fact, it suffices to assume that $d_l^k, \lambda_l^k\in\mathbb{R}^{|e_l|\times |e_l|}$, since
\begin{equation}\label{eq:alg1}
  d_l^k(x_i,x_j)=\lambda_l^k(x_i,x_j)=0,
\end{equation}
whenever either $x_i$ or $x_j$ belongs to $V\backslash e_l$.

If $m = 1$, Algorithm \ref{alg1} coincides with the classical ADMM with an additional proximal point term in the $u$-subproblem, also known as the proximal point algorithm applied to the augmented Lagrangian formulation \cite{rockafellar1976augmented,zhang2011unified}.
When $F$ is not strictly convex, the proximal point term compensates for the lack of strict convexity of the objective function and results in a more stable algorithm.
By this observation, Algorithm \ref{alg1} can be regarded as a proximal ADMM with incremental updates
for the constrained optimization problem
\begin{equation}\label{eq:3.7a}
  \sum_{k=1}^{m}\|d_k\|_\infty^p + F(u) \quad \mbox{s.t. } d_k=D_ku,
\end{equation}
The associated Lagrangian reads
\begin{equation}\label{eq:Lagrangian}
  \mathcal{L}(u,d_1\cdots d_m;\lambda_1\cdots \lambda_m)
  =F(u) + \sum_{k=1}^{m}\|d_k\|_\infty^p + \langle \lambda_k, D_ku-d_k \rangle,
\end{equation}
where $\lambda_i$, $i=1,\cdots, m$, are the Lagrangian multipliers.

\begin{remark}\label{remark:1}
We use zero initialization for both $d_i$ and $\lambda_i$ in Algorithm \ref{alg1}.
During the first cycle of iterations, i.e., $0\leq k\leq m-1$,
it follows from \eqref{eq:3.8} that
  \[
  u^{k+1}=\arg\min_{u\in\mathbb{R}^n}
  \frac{1}{m} F(u)
  +\frac{1}{2\tau} \|u-u^k\|^2+\frac{\rho}{2}\|D_{k+1}u\|^2,
  \]
which is exactly the Laplace learning on hyperedge $e_{k+1}$ with two additional constraints \cite{zhu2003semi}.
This observation provides an intuitive interpretation for the minimizer deduced by Algorithm \ref{alg1}, which can be viewed as a correction to classical Laplace learning.
\end{remark}

\begin{remark}\label{re:2}
  We assume that $F$ is split uniformly into $m$ parts in the operator-splitting step of \eqref{eq:3.5} In fact, any partition is admissible.
  A particular choice is the partition $(F, 0, \dots, 0)$ associated with the hyperedges $(e_1, e_2, \dots, e_m)$.
  This choice is beneficial when subproblem \eqref{eq:3.8} is difficult to solve in the presence of $F$, but becomes significantly easier when $F$ is omitted.
\end{remark}

\subsection{Solutions to subproblems \eqref{eq:3.8} and \eqref{eq:3.9}}

The $u$-subproblem \eqref{eq:3.8} depends on the data fidelity term $F$.
We consider $F(u)=\mathbb{I}_L(u)$.
According to \eqref{eq:alg1}, subproblem \eqref{eq:3.8} can be divided into two parts. The solution on $V\backslash e_{i_k}$ is trivial.
Namely,
\begin{align*}
  u^{k+1}(x_i)=
    u^{k}(x_i), \quad \mbox{if } x_i\in V\backslash e_{i_k}.
\end{align*}
For the solution on $e_{i_k}$,
let $v^{k}=d_{i_k}^{k}-\frac{\lambda_{i_k}^{k}}{\rho}$. Then
\begin{align*}
  u^{k+1}|_{e_{i_k}}
  &=\arg\min_{u\in\mathbb{R}^{|e_{i_k}|}}\frac{\rho}{2}
  \left\|D_{i_k}u-v^{k}\right\|^2
  +\frac{1}{2\tau} \|u-u^k\|^2+\frac{1}{m}\mathbb{I}_L(u)\\
  &=\arg\min_{u\in\mathbb{R}^{|e_{i_k}|}}
  \sum_{x_i,x_j\in e_{i_k}}\left|w^{\frac{1}{p}}_{i,j}(u(x_i)-u(x_j))-v^{k}_{i,j}\right|^2
  +\frac{1}{\rho\tau} \|u-u^k\|^2+\mathbb{I}_L(u).
\end{align*}
By the optimality condition,
$u^{k+1}|_{e_{i_k}}$ satisfies the linear system
\begin{equation*}
  \begin{cases}
    \sum\limits_{x_j\in e_{i_k}}\left(w_{i,j}^{\frac{2}{p}}+w_{j,i}^{\frac{2}{p}}\right)
    (u(x_i)-u(x_j))
    +w^{\frac{1}{p}}_{j,i}v^{k}_{j,i}-w^{\frac{1}{p}}_{i,j}v^{k}_{i,j}
    +\frac{1}{\rho\tau}(u(x_i)-u^k(x_i))=0,  &\mbox{if } x_i\in e_{i_k}\backslash L, \\
    u(x_i)=y_i, &\mbox{if } x_i\in e_{i_k}\cap L.
  \end{cases}
\end{equation*}
The coefficient matrix of the system is symmetric and positive definite.
If $w_{i,j} = 1$ on the hyperedge $e_{i_k}$, the system admits a closed-form solution.
More precisely,
\[
u(x_i)=\frac{1}{2|e_{i_k}|+\frac{1}{\rho\tau}}\left(h(x_i)
+\frac{1}{\frac{1}{2\rho\tau}+|e_{i_k}\cap L|}
\sum_{x_i\in e_{i_k}\backslash L}h(x_i)\right), \quad x_i\in e_{i_k}\backslash L,
\]
where
\[
h(x_i)=2\sum_{x_j\in e_{i_k}\cap L}u(x_j)
+\sum_{x_j\in e_{i_k}}({v^{k}_{i,j}-v^{k}_{j,i}})
+\frac{1}{\rho\tau}u^k(x_i).
\]

For the $d$-subproblem \eqref{eq:3.9},
let us recall the definition of the proximal operator
  \begin{equation*}
  \mbox{Prox}_{f}(\alpha)=\arg\min_{\beta} f(\beta)+ \frac{1}{2}\|\beta-\alpha\|^2.
\end{equation*}
The it can be rewritten as
\begin{align*}
  d_{i_k}^{k+1}
  =\mbox{Prox}_{\frac{1}{\rho}\|\cdot\|^p_\infty}
  \left(D_{i_k}u^{k+1}+\frac{\lambda_{i_k}^{k}}{\rho}\right)
  =:\mbox{Prox}_{g}
  \left(\beta\right).
\end{align*}
It follows from the well-known Moreau decomposition
$\beta=\mbox{Prox}_{g}\left(\beta\right)+\mbox{Prox}_{g^*} \left(\beta\right)$
that
\begin{equation}\label{eq:2.3a}
  d_{i_k}^{k+1}=\mbox{Prox}_{g}\left(\beta\right)
  =\beta-\mbox{Prox}_{g^*}\left(\beta\right),
\end{equation}
where
\begin{equation*}
  g^*(\beta)=\max_{\alpha}\langle \alpha,\beta\rangle -g(\alpha)
\end{equation*}
denotes the Fenchel conjugate of $g$.
We are left to compute $g^*$ and $\mbox{Prox}_{g^*}\left(\beta\right)$.

If $p=1$, by the definition of the Fenchel conjugate, we have
\begin{equation*}
  g^*(\gamma)=
  \begin{cases}
    0, & \mbox{if } \|\gamma\|_1\leq \frac{1}{\rho}, \\
    +\infty, & \mbox{otherwise}.
  \end{cases}
\end{equation*}
Then
\[
  \mbox{Prox}_{g^*}\left(\beta\right)
  =\arg\min_{\gamma}
  g^*(\gamma)
  +\frac{1}{2}\|\gamma-\beta\|^2
\]
is just the projection onto the $\ell^1$ ball, denoted by $\mbox{Proj}_{\|\cdot\|_1\leq \frac{1}{\rho}}(\beta)$.
Consequently,
\begin{align*}
  d_{i_k}^{k+1}
  =\beta-\mbox{Prox}_{g^*}
  \left(\beta\right)
  =\left(D_{i_k}u^{k+1}+\frac{\lambda_{i_k}^{k}}{\rho}\right)
  -\mbox{Proj}_{\|\cdot\|_1\leq \frac{1}{\rho}}\left(D_{i_k}u^{k+1}+\frac{\lambda_{i_k}^{k}}{\rho}\right).
\end{align*}
Here $\mbox{Proj}_{\|\cdot\|_1\leq \frac{1}{\rho}}$ can be solved efficiently \cite{condat2016fast}.

For the case $p>1$, we have the following Lemma \cite[Lemma 5.1]{shi2025hypergraph}.
\begin{lemma}
    Let $p>1$ and $\rho>0$. Then
    \begin{equation*}
      g^*(\gamma)
      =\left(\frac{\rho}{p}\right)^{p'-1}\frac{1}{p'}\|\gamma\|_1^{p'},
    \end{equation*}
    where $p'=\frac{p}{p-1}$.
\end{lemma}
It follows that
\begin{equation}\label{eq:2.4}
  \mbox{Prox}_{g^*}
  \left(\beta\right)
  =\arg\min_\gamma\left(\frac{\rho}{p}\right)^{p'-1}\frac{1}{p'}\|\gamma\|_1^{p'}
  +\frac{1}{2}\|\gamma-\beta\|^2.
\end{equation}
Since $p'>1$ and the $p'$-th power of the $\ell^1$ norm is non-separable, problem \eqref{eq:2.4} is more subtle.
We adopt the algorithm of \cite{prater2023constructive}.
After obtaining $\mbox{Prox}_{g^*}(\beta)$, we substitute it into \eqref{eq:2.3a} to find the solution of the $d$-subproblem $d_{i_k}^{k+1}$.

\section{Convergence analysis}
The convergence of Algorithm \ref{alg1} can be viewed as a generalization of the proximal ADMM \cite{zhang2011unified} to the large-scale setting.
It relies on the convexity of the data fidelity term $F$.

\begin{lemma}\label{le:convex}
Let $F:\mathbb{R}^n\rightarrow \mathbb{R}$ be convex. For any $u_1,\cdots, u_l\in\mathbb{R}^n$, there exists a $u\in \mathbb{R}^n$, such that
\begin{equation*}
  F(u)=\frac{1}{l}\sum_{i=1}^l F(u_i).
\end{equation*}
\end{lemma}
\begin{proof}
Let $C=\operatorname{conv}\{u_1,\cdots, u_l\}$ denote the convex hull of the finite set
$\{u_1,\cdots, u_l\}$. Clearly, $C\subset\mathbb{R}^n$ is compact and connected.
It follows from the convexity of $F$ that $F$ is continuous on $C$.
Consequently, the image set $F(C)=\{F(u): u\in C\}$ is a compact and connected subset
of $\mathbb{R}$ (i.e., a closed interval).

Observe that
\[
\min_{1\le i\le l} F(u_i)\;\le\;\bar F:=\frac{1}{l}\sum_{i=1}^l F(u_i) \;\le\;\max_{1\le i\le l} F(u_i).
\]
In particular $\bar F\in [\min_i F(u_i),\max_i F(u_i)]= F(C)$.
By the definition of the image set, there exists a $u\in C$ with
$F(u)=\bar F$. This proves the lemma.
\end{proof}

The convergence of Algorithm \ref{alg1} is stated as follows.
\begin{theorem}
Assume that $F:\mathbb{R}^n\rightarrow \mathbb{R}$ is convex.
  Let the sequence $(u^k, d_1^k, \cdots, d_m^k, \lambda_1^k,\cdots,\lambda_m^k)$ be generated by Algorithm \ref{alg1}. Then
  \begin{equation}\label{eq:th3:1}
  \mbox{ $(u^k, d_1^k, \cdots, d_m^k, \lambda_1^k,\cdots,\lambda_m^k)$ is bounded; \quad $u^k$, $d_i^k$, and $D_iu^{k}-d_i^k$ are convergent}
  \end{equation}
  for any $i=1,\cdots, m$,
  \begin{equation}\label{eq:th3:2}
    \lim_{k\rightarrow\infty}\sum_{i=1}^{m}\|d_i^{k}\|_\infty^p + F(u^k)
    =\inf_{u\in\mathbb{R}^n}J(u),
  \end{equation}
  and all limit points of $(u^k, d_1^k, \cdots, d_m^k, \lambda_1^k,\cdots,\lambda_m^k)$ are saddle points of $\mathcal{L}(u,d_1\cdots d_m;\lambda_1\cdots \lambda_m)$, defined in \eqref{eq:Lagrangian}.
\end{theorem}

\begin{proof}

Let $(\bar{u},\bar{d}_1\cdots \bar{d}_m;\bar{\lambda}_1\cdots \bar{\lambda}_m)$ be a saddle point of the Lagrangian $\mathcal{L}(u,d_1\cdots d_m;\lambda_1\cdots \lambda_m)$.
We have
\begin{equation}\label{eq:3.12}
  D_i\bar{u}-\bar{d}_i=0, \quad  i=1,2,\cdots m.
\end{equation}
Furthermore,
\begin{equation}\label{eq:3.11}
  \mathcal{L}(\bar{u},\bar{d}_1\cdots \bar{d}_m;\bar{\lambda}_1\cdots \bar{\lambda}_m)
  \leq \mathcal{L}({u},{d}_1 \cdots {d}_m;\bar{\lambda}_1\cdots \bar{\lambda}_m),
\end{equation}
for any $u\in\mathbb{R}^n, d_i\in\mathbb{R}^{n^2}, i=1,\cdots,m$.

Let $k\geq 0$ and $i_k=\mbox{mod}(k,m)+1$.
Applying the optimality condition for \eqref{eq:3.8}--\eqref{eq:3.9} and repeating \eqref{eq:3.10}, we have
\begin{align}\label{eq:3.13}
  \begin{cases}
    \frac{1}{\tau}(u^{k+1}-u^k)+\frac{1}{m} q^{k+1}+D_{i_k}^T\lambda_{i_k}^k+\rho D_{i_k}^T(D_{i_k}u^{k+1}-d_{i_k}^{k}) =0,\\
    p_{i_k}^{k+1}-\lambda_{i_k}^k-\rho(D_{i_k}u^{k+1}-d_{i_k}^{k+1})=0,\\
    \lambda^{k+1}_{i_k}=\lambda_{i_k}^{k}+\rho(D_{i_k}u^{k+1}-d_{i_k}^{k+1}),
  \end{cases}
\end{align}
where $q^{k+1}\in \partial F(u^{k+1})$ and $p_{i_k}^{k+1}\in \partial \|d_{i_k}^{k+1}\|_\infty^p$.
Denote
\[
u^k_e=u^k-\bar{u},\quad
d^k_{i,e}=d^k_i - \bar{d}_i,\quad
\lambda^k_{i,e}=\lambda^k_i - \bar{\lambda}_i.
\]
Then by \eqref{eq:3.12} and \eqref{eq:3.13},
\begin{align}\label{eq:3.14}
  \begin{cases}
    \frac{1}{\tau}(u^{k+1}_e-u^k_e)+\frac{1}{m} q^{k+1}+D_{i_k}^T\lambda_{i_k}^k+\rho D_{i_k}^T(D_{i_k}u^{k+1}_e-d_{i_k,e}^{k}) =0,\\
    p_{i_k}^{k+1}-\lambda_{i_k}^k-\rho(D_{i_k}u^{k+1}_e-d_{i_k,e}^{k+1})=0,\\
    \frac{1}{\rho}(\lambda^{k+1}_{i_k}-\lambda_{i_k}^{k})-(D_{i_k}u_e^{k+1}-d_{i_k,e}^{k+1})=0,
  \end{cases}
\end{align}
Taking inner products of \eqref{eq:3.14} with $u_e^{k+1}$, $d_{i_k,e}^{k+1}$, and $\lambda_{i_k}^{k}$ and summing the results yields
\begin{align}\label{eq:3.15}
\begin{split}
  \frac{1}{\tau}\langle u_e^{k+1}, u_e^{k+1}-u_e^k\rangle
  &+\frac{1}{m}\langle u_e^{k+1}, q^{k+1}\rangle
  +\langle d_{i_k,e}^{k+1}, p_{i_k}^{k+1}\rangle
  +\frac{1}{\rho}\langle\lambda^{k}_{i_k}, \lambda^{k+1}_{i_k}-\lambda^{k}_{i_k} \rangle\\
   &=\rho \langle D_{i_k}u_e^{k+1},  d_{i_k,e}^{k}-D_{i_k}u_e^{k+1}\rangle
   -\rho \langle d_{i_k,e}^{k+1},  d_{i_k,e}^{k+1}-D_{i_k}u_e^{k+1}\rangle.
\end{split}
\end{align}
By the definition of subdifferential for $F$ and $\|\cdot\|_\infty^p$,
\begin{align*}
    -&\frac{1}{m}\langle u_e^{k+1}, q^{k+1}\rangle
    -\langle d_{i_k,e}^{k+1}, p_{i_k}^{k+1}\rangle
    \leq \frac{1}{m}F(\bar{u})-\frac{1}{m}F(u^{k+1})
   +\|\bar{d}_{i_k}\|_\infty^p - \|d_{i_k}^{k+1}\|_\infty^p  \\
   &=
   \underbrace{\frac{1}{m}F(\bar{u})-\frac{1}{m}F(u^{k+1})
   +\|\bar{d}_{i_k}\|_\infty^p - \|d_{i_k}^{k+1}\|_\infty^p
   - \langle \bar{\lambda}_{i_k}, D_{i_k}u^{k+1}-d_{i_k}^{k+1} \rangle}_{=:A_{i_k}^{k+1}}
   + \langle \bar{\lambda}_{i_k}, D_{i_k}u^{k+1}-d_{i_k}^{k+1} \rangle.
\end{align*}
Substituting it into \eqref{eq:3.15} and recalling the third equality in \eqref{eq:3.13} to find
\begin{align}\label{eq:3.16}
\begin{split}
  \frac{1}{\tau}\langle u_e^{k+1}, u_e^{k+1}&-u_e^k\rangle
  +\frac{1}{\rho}\langle\lambda^{k}_{i_k,e}, \lambda^{k+1}_{i_k,e}-\lambda^{k}_{i_k,e} \rangle\\
   &\leq \rho \langle D_{i_k}u_e^{k+1},  d_{i_k,e}^{k}-D_{i_k}u_e^{k+1}\rangle
   -\rho \langle d_{i_k,e}^{k+1},  d_{i_k,e}^{k+1}-D_{i_k}u_e^{k+1}\rangle
   + A_{i_k}^{k+1}.
\end{split}
\end{align}
For the left-hand side of \eqref{eq:3.16},
\[
\frac{1}{\tau}\langle u_e^{k+1}, u_e^{k+1}-u_e^k\rangle
= \frac{1}{2\tau}\|u_e^{k+1}\|^2
- \frac{1}{2\tau}\|u_e^{k}\|^2
+\frac{1}{2\tau}\|u^{k+1}-u^{k}\|^2,
\]
and
\[
\frac{1}{\rho}\langle\lambda^{k}_{i_k,e}, \lambda^{k+1}_{i_k,e}-\lambda^{k}_{i_k,e} \rangle
=\frac{1}{2\rho}(\|\lambda^{k+1}_{i_k,e}\|^2-\|\lambda^{k}_{i_k,e}\|^2)-\frac{\rho}{2}\|D_{i_k}u_e^{k+1}-d_{i_k,e}^{k+1}\|^2,
\]
where we use the third equality in \eqref{eq:3.14}.
The right-hand side of \eqref{eq:3.16}
\[
=\frac{\rho}{2}\left(\|d_{i_k,e}^{k}\|_2^2-\|d_{i_k,e}^{k+1}\|^2\right)-\frac{\rho}{2}\|D_{i_k}u_e^{k+1}-d_{i_k,e}^{k}\|^2
  -\frac{\rho}{2}\|D_{i_k}u_e^{k+1}-d_{i_k,e}^{k+1}\|^2 + A_{i_k}^{k+1},
\]
where by \eqref{eq:3.12},
\begin{align*}
\|D_{i_k}u_e^{k+1}-d_{i_k,e}^{k}\|^2
=\|D_{i_k}u^{k+1}-d_{i_k}^{k}\|^2
=\|D_{i_k}u^{k+1}-d_{i_k}^{k+1}\|^2
&+\|d_{i_k}^{k+1}-d_{i_k}^{k}\|^2 \\
&+2\langle D_{i_k}u^{k+1}-d_{i_k}^{k+1}, d_{i_k}^{k+1}-d_{i_k}^{k}\rangle,
\end{align*}
Substituting all the results into \eqref{eq:3.16}, we have
\begin{align}\label{eq:3.17}
\begin{split}
  \frac{1}{2\tau}\|u_e^{k+1}\|^2 +\frac{\rho}{2}\|d_{i_k,e}^{k+1}\|^2
  + \frac{1}{2\rho}\|\lambda^{k+1}_{i_k,e}\|^2
  +\frac{1}{2\tau}\|u^{k+1}-u^{k}\|^2
  +\frac{\rho}{2}\|D_{i_k}u^{k+1}-d_{i_k}^{k+1}\|^2
  +\frac{\rho}{2}\|d_{i_k}^{k+1}-d_{i_k}^{k}\|^2 \\
  \leq \frac{1}{2\tau}\|u_e^{k}\|^2 +\frac{\rho}{2}\|d_{i_k,e}^{k}\|^2
  +\frac{1}{2\rho}\|\lambda^{k}_{i_k,e}\|^2
  -\rho\langle D_{i_k}u^{k+1}-d_{i_k}^{k+1}, d_{i_k}^{k+1}-d_{i_k}^{k}\rangle
  +A_{i_k}^{k+1}.
\end{split}
\end{align}

Let us consider an iteration cycle, i.e., $k=cm+i-1$ for $i=1,\cdots,m$, where $c\geq 0$ is an integer.
Clearly, according to the update rule for $d$ and $\lambda$ in Algorithm \ref{alg1},
\[
d_{i_k}^k=d_{i}^{cm+i-1}=d_{i}^{(c-1)m+i},\quad
\lambda_{i_k}^k=\lambda_{i}^{cm+i-1}=\lambda_{i}^{(c-1)m+i}.
\]
Here we assume $(c-1)m+i=0$ when $c=0$.
With the new notation, \eqref{eq:3.17} becomes
\begin{align}\label{eq:3.18}
\begin{split}
  \frac{1}{2\tau}\|u_e^{cm+i}\|^2 +\frac{\rho}{2}\|d_{i,e}^{cm+i}\|^2
  + \frac{1}{2\rho}\|\lambda^{cm+i}_{i,e}\|^2
  +\frac{1}{2\tau}\|u^{cm+i}-u^{cm+i-1}\|^2
  +\frac{\rho}{2}\|D_{i}u^{cm+i}-d_{i}^{cm+i}\|^2 \\
  +\frac{\rho}{2}\|d_{i}^{cm+i}-d_{i}^{(c-1)m+i}\|^2
  \leq \frac{1}{2\tau}\|u_e^{cm+i-1}\|^2 +\frac{\rho}{2}\|d_{i,e}^{(c-1)m+i}\|^2
  +\frac{1}{2\rho}\|\lambda^{(c-1)m+i}_{i,e}\|^2\\
  -\rho\langle D_{i}u^{cm+i}-d_{i}^{cm+i}, d_{i}^{cm+i}-d_{i}^{(c-1)m+i}\rangle
  +A_i^{cm+i}.
\end{split}
\end{align}
We estimate the last two terms of \eqref{eq:3.18}.
By the second and the third equalities of \eqref{eq:3.13} and the convexity of $\|\cdot\|_\infty^p$, the inner product on the right-hand side of \eqref{eq:3.18}
\[
=\frac{1}{\rho}\langle \lambda_i^{cm+i}-\lambda_i^{(c-1)m+i}, d_{i}^{cm+i}-d_{i}^{(c-1)m+i}\rangle
=\frac{1}{\rho}\langle p_i^{cm+i}-p_i^{(c-1)m+i}, d_{i}^{cm+i}-d_{i}^{(c-1)m+i}\rangle
\geq 0.
\]
To estimate $\sum_{i=1}^{m} A_i^{cm+i}$,
we observe from Lemma \ref{le:convex} that there exists a $u\in\mathbb{R}^n$ such that
\[
\sum_{i=1}^{m}
\frac{1}{m}F(u^{cm+i})
-\frac{1}{m}F(u)
+\langle \bar{\lambda}_{i}, D_{i}u^{cm+i}-D_{i}u \rangle
=0.
\]
Consequently, by \eqref{eq:3.12} and \eqref{eq:3.11},
\begin{align*}
 \sum_{i=1}^{m} A_i^{cm+i}
 &=\sum_{i=1}^{m}\frac{1}{m}F(\bar{u})-\frac{1}{m}F(u^{cm+i})
   +\|\bar{d}_{i}\|_\infty^p - \|d_{i}^{cm+i}\|_\infty^p
   - \langle \bar{\lambda}_{i}, D_{i}u^{cm+i}-d_{i}^{cm+i} \rangle\\
&=\sum_{i=1}^{m}\frac{1}{m}F(\bar{u})-\frac{1}{m}F(u)
   +\|\bar{d}_{i}\|_\infty^p - \|d_{i}^{cm+i}\|_\infty^p
   - \langle \bar{\lambda}_{i}, D_{i}u-d_{i}^{cm+i} \rangle\\
   &=\mathcal{L}(\bar{u},\bar{d}_1\cdots \bar{d}_m;\bar{\lambda}_1\cdots \bar{\lambda}_m)
  - \mathcal{L}({u},{d}_1^{cm+1} \cdots {d}_m^{cm+m};\bar{\lambda}_1\cdots \bar{\lambda}_m)
  \leq 0.
\end{align*}

By summing \eqref{eq:3.18} first over $i = 1, \dots, m$ and then over $c = 0, \dots, \infty$,
we obtain that the sequence $(u^{(c+1)m}, d_i^{cm+i}, \lambda_i^{cm+i})$ is bounded for any $i=1,\cdots,m$, $c\geq 0$
and
\[
\lim_{c\rightarrow\infty}\|u^{cm+i}-u^{cm+i-1}\|^2=0,\quad
\lim_{c\rightarrow\infty}\|d_{i}^{cm+i}-d_{i}^{(c-1)m+i}\|^2=0,\quad
\lim_{c\rightarrow\infty}\|D_{i}u^{cm+i}-d_{i}^{cm+i}\|^2=0,
\]
for any $i=1,\cdots,m$.
Then the sequence $(u^k, d_1^k,\cdots, d_m^k,\lambda_1^k,\cdots,\lambda_m^k)$ is bounded for any $k\geq 0$ and
\[
\lim_{k\rightarrow\infty}\|u^{k+1}-u^{k}\|^2
    =\lim_{k\rightarrow\infty}\|d_i^{k+1}-d_i^{k}\|^2
    =\lim_{k\rightarrow\infty}\|D_iu^{k}-d_i^k\|^2
    =0.
\]
This proves \eqref{eq:th3:1}.

By the boundedness of $(u^k, d_1^k,\cdots, d_m^k,\lambda_1^k,\cdots,\lambda_m^k)$ and \eqref{eq:3.13}, $q^{k+1}$ and $p_{i_k}^{k+1}$ are also bounded.
There exist convergent subsequences (still denoted by themselves) and cluster points
\[
(u^\infty, d_1^\infty, \cdots, d_m^\infty, \lambda_1^\infty,\cdots,\lambda_m^\infty)\quad \mbox{and}\quad (q^\infty,p_1^\infty,\cdots,p_m^\infty).
\]
Clearly,
\[
q^\infty\in\partial F(u^\infty)\quad \mbox{and } \quad
p_i^\infty\in \partial \|d_i^\infty\|^p_\infty, ~i=1,\cdots, m.
\]
Passing to the limit $k\rightarrow\infty$ in \eqref{eq:3.13} to find
\begin{align*}
  \begin{cases}
    \frac{1}{m}q^\infty+D_i^T\lambda_i^\infty=0,\quad &i=1,\cdots, m, \\
    p_i^\infty-\lambda_i^\infty=0, \quad &i=1,\cdots, m,\\
    D_iu^\infty-d_i^\infty=0,\quad &i=1,\cdots, m.
  \end{cases}
\end{align*}
Thus $(u^\infty, d_1^\infty, \cdots, d_m^\infty, \lambda_1^\infty,\cdots,\lambda_m^\infty)$ is a saddle point of $\mathcal{L}(u,d_1\cdots d_m;\lambda_1\cdots \lambda_m)$.

Notice that the convergence $u^k\rightarrow u^\infty$ and $d_i^k\rightarrow d_i^\infty$, $i=1,\cdots,m$, hold for the whole sequence (not just a subsequence).
We have
\begin{align*}
\lim_{k\rightarrow\infty} F(u^k)+\sum_{i=1}^{m}\|d_i^{k}\|_\infty^p
=F(u^\infty)+\sum_{i=1}^{m}\|d_i^\infty\|_\infty^p
=\mathcal{L}(u^\infty, d_1^\infty, \cdots, d_m^\infty; \lambda_1^\infty,\cdots,\lambda_m^\infty)
=\inf_{u\in\mathbb{R}^n}J(u).
\end{align*}
This proves \eqref{eq:th3:2}.
\end{proof}

\section{Numerical experiments}
In this section, we present applications of the hypergraph $p$-Laplacian \eqref{eq:1.1}--\eqref{eq:1.1a} (HpL) to missing data recovery.
Specifically, we set $F(u) = \mathbb{I}_L(u)$ and consider the tasks of image sparse inpainting and semi-supervised learning.
For comparison, we use the graph Laplacian (GL), i.e., \eqref{eq:GL} with $p = 2$. Missing data recovery using GL is equivalent to solving a symmetric positive definite linear system, which is typically sparse and can be solved efficiently.
For HpL, we compare the proposed Algorithm \ref{alg1} with the stochastic primal-dual hybrid gradient algorithm (sPDHG) \cite{shi2025hypergraph} and the subgradient descent method (SGD) \cite{zhang2017re}.

In Algorithm \ref{alg1}, two parameters, $\tau$ and $\rho$, need to be tuned. We adopt a diminishing step size $\tau_{c+1} = \tau_c / r$ at the $(c+1)$-th iteration cycle and set $\tau_0 = 1$.
The parameters $\rho$ and $r$ depend on the task and will be specified below.
All experiments are conducted in MATLAB 2020b on a desktop equipped with an Intel Core i7-8700 3.20 GHz CPU.
The code implementing the experiments is available at \url{https://github.com/kshi13/OS-HpL}.

\subsection{Image sparse inpainting}\label{section:inpainting}
In image sparse inpainting, we are given the intensity $f_{i,j}$ of a grayscale image $f\in \mathbb{R}^{N_1\times N_2}$ on some sparse pixels $(i,j)\in L\subset \{(i,j): 1\leq i\leq N_1, 1\leq j\leq N_2\}$. The goal is to interpolate the image $f$ based on these partially observed pixels.

To employ the graph/hypergraph model for image processing, a graph/hypergraph must first be constructed from the given image. The construction is not unique and may vary depending on the specific application.
Here we follow the approach of \cite{shi2017weighted,shi2025hypergraph}, which treats each pixel as a vertex and constructs edges or hyperedges based on the semilocal similarity between pixels.
Specifically, we define a point cloud from the image $f$ by
\begin{equation*}
  \Omega_n(f)=\left\{\bar{P}_{i,j}(f): 1\leq i\leq N_1, 1\leq j\leq N_2\right\}\subset \mathbb{R}^{s_1\times s_2},
\end{equation*}
where
\begin{equation*}
  \bar{P}_{i,j}(f)=[P_{i,j}(f), \lambda\bar{x}], \quad \bar{x}=\left({i}/{N_1}, {j}/{N_2}\right),
\end{equation*}
includes the local coordinate $\bar{x}$
and $P_{i,j}(f)$ is the rectangle image patch centered at pixel $(i,j)$ with size $s_1\times s_2$.
Using the $\ell^2$ norm in $\R^{s_1\times s_2 +2}$ for $\Omega_n(f)$,
the weight function between two patches $\bar{P}_{i,j}(f)$ and $\bar{P}_{k,l}(f)$ is defined by
\begin{equation}\label{eq:weight}
  w_{\bar{P}_{i,j}(f),\bar{P}_{k,l}(f)}=\exp\left(-\frac{\|\bar{P}_{i,j}(f)-\bar{P}_{k,l}(f)\|^2}
  {\sigma(\bar{P}_{i,j}(f))^2}\right),
\end{equation}
where $\sigma(\bar{P}{i,j}(f))$ denotes the distance between $\bar{P}_{i,j}(f)$ and its $k_n$-th nearest neighbor.
Finally, a $k_n$-nearest neighbor (k-NN) graph/hypergraph is constructed from the point cloud.
Let $u$ be a function on $\Omega_n(f)$ such that
\begin{equation*}
  u(\bar{P}_{i,j}(f))=f_{i,j}.
\end{equation*}
The image inpainting problem is then reformulated as finding the function $u$ on $\Omega_n(f)$.

\begin{algorithm}[tbp]
  \caption{Image sparse inpainting.}
  \label{alg3}
  \begin{algorithmic}
  \REQUIRE Partially observed pixels $\{f_{i,j},(i,j)\in L\}$, initial guess $u^0$.
  \FOR { $k=0:K-1$}
  \STATE Build point cloud $\Omega_n(u^k)$, calculate the weight \eqref{eq:weight}, and build the graph/hypergraph.
  \STATE Construct the training set $\{(\bar{P}_{i,j}(u^k), f_{i,j}): (i,j)\in L\}$.
  \STATE Solve $u^{k+1}$ on $\Omega_n(u^k)$ by the training set.
  \ENDFOR
  \RETURN $u^{K-1}$.
  \end{algorithmic}
\end{algorithm}

\begin{figure}[tbp]
  \centering
  \begin{tabular}{@{}c@{~}c@{~}c@{~}c@{~}c@{~}c@{}}
    \includegraphics[width=.1\textwidth]{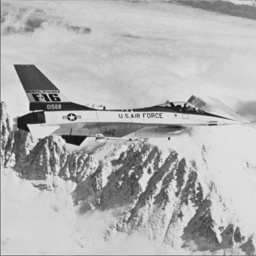} &
    \includegraphics[width=.1\textwidth]{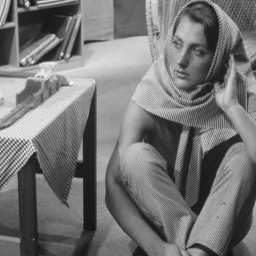} &
    \includegraphics[width=.1\textwidth]{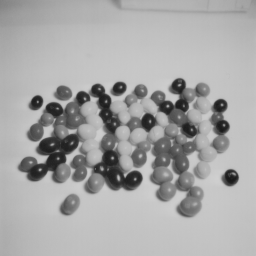} &
    \includegraphics[width=.1\textwidth]{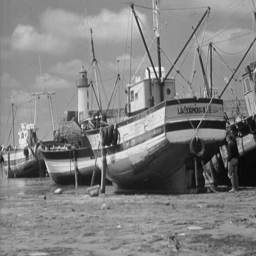} &
    \includegraphics[width=.1\textwidth]{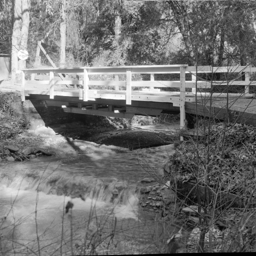} &
    \includegraphics[width=.1\textwidth]{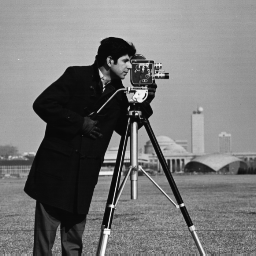} \\
    \footnotesize\emph{Airplane} & \footnotesize\emph{Barbara}  & \footnotesize\emph{Beans}  & \footnotesize\emph{Boat}  & \footnotesize\emph{Bridge}  & \footnotesize\emph{Cameraman} \\
    \includegraphics[width=.1\textwidth]{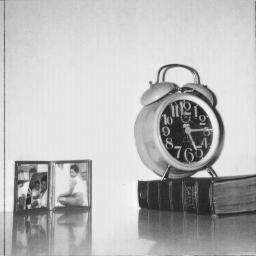} &
    \includegraphics[width=.1\textwidth]{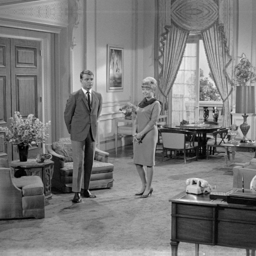} &
    \includegraphics[width=.1\textwidth]{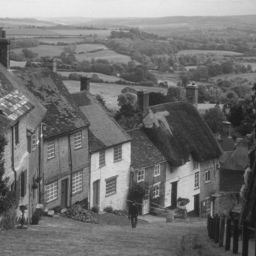} &
    \includegraphics[width=.1\textwidth]{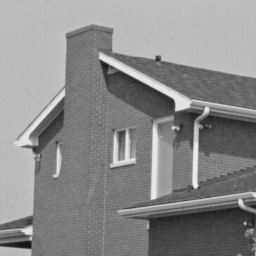} &
    \includegraphics[width=.1\textwidth]{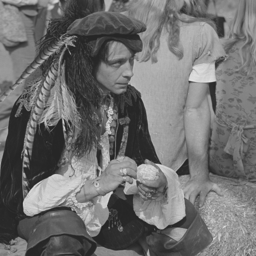} &
    \includegraphics[width=.1\textwidth]{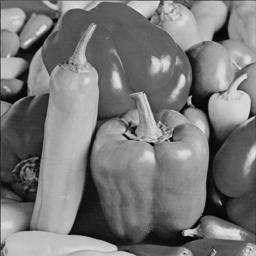} \\
    \footnotesize\emph{Clock} & \footnotesize\emph{Couple} & \footnotesize\emph{Hill} & \footnotesize\emph{House} & \footnotesize\emph{Man} & \footnotesize\emph{Peppers} \\
  \end{tabular}
    \caption{Test images ($256\times256$) for image inpainting.}
    \label{fig:testimage}
  \end{figure}

In image inpainting, the original image $f$ is unknown. We begin with an initial guess of $f$ and iteratively update it using the restored image. See Algorithm \ref{alg3} for details.
Throughout this section, we set $k_n=21$, $s_1=s_2=11$, and $\lambda=10$.
For GL, the number of iterations in Algorithm \ref{alg3} is chosen as $K = 15$.
For the hypergraph $p$-Laplacian \eqref{eq:1.1}--\eqref{eq:1.1a}, we initialize it with the solution obtained from GL and set $K = 1$.
The test images are shown in Figure \ref{fig:testimage}.

Before proceeding with image inpainting experiments, we first test HpL with $p=1$ and $p=2$, i.e., the hypergraph total variation and the hypergraph Laplacian.
We assume that the original image $f$ is known, which allows us to compute the ground truth weights \eqref{eq:weight}. Taking Barbara as the test image, Figure \ref{fig:convergence} presents the inpainting results and the corresponding energy for $p=1,2$ under sampling rates of $s=15\%, 5\%$, respectively.
Although the case $p=1$ and $p=2$ are theoretically distinct, we do not observe a significant difference in the results. When $p=1$, a slightly higher PSNR is achieved.
Unlike the case of graph or continuum PDEs, HpL is non-differentiable for any $p\geq 1$.
Algorithm \ref{alg1} is the fastest when $p=1$, since the corresponding $d$-subproblem is the easiest to solve.
In the subsequent experiments, for the sake of comparison, we set $p=2$.

\begin{figure}[tbp]
  \centering
  \begin{tabular}{@{}c@{~}c@{~~~}c@{~}c@{}}
  \includegraphics[width=.2\textwidth]{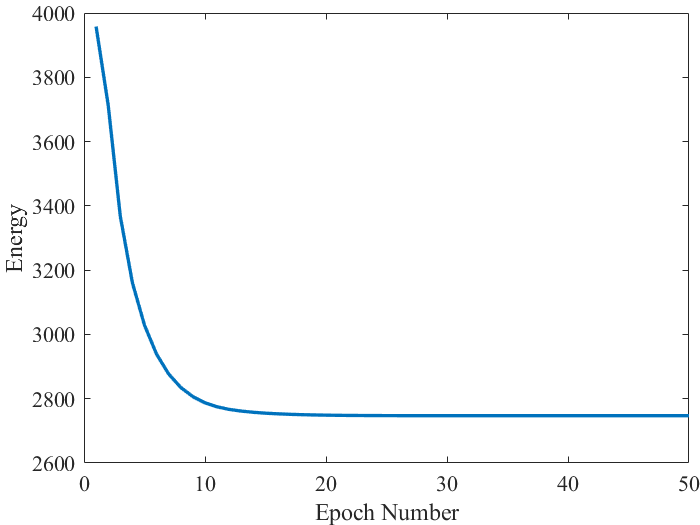}&
  \includegraphics[width=.2\textwidth]{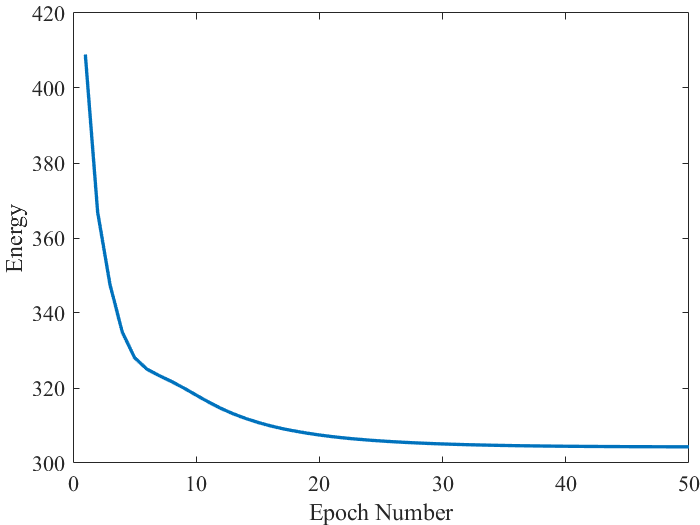}&
  \includegraphics[width=.2\textwidth]{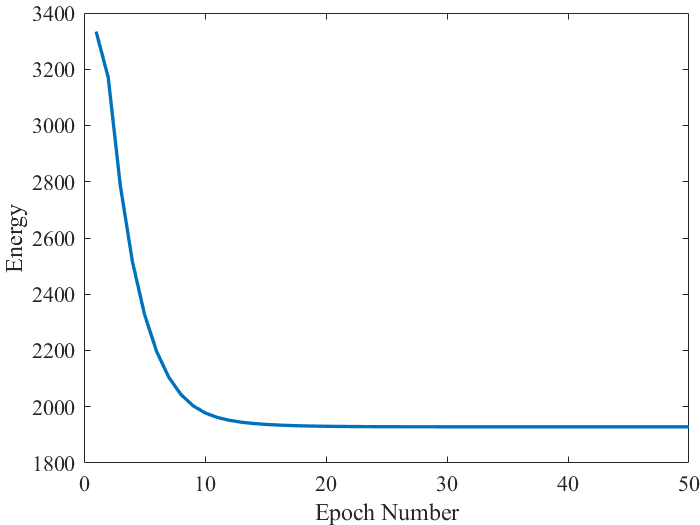}&
  \includegraphics[width=.2\textwidth]{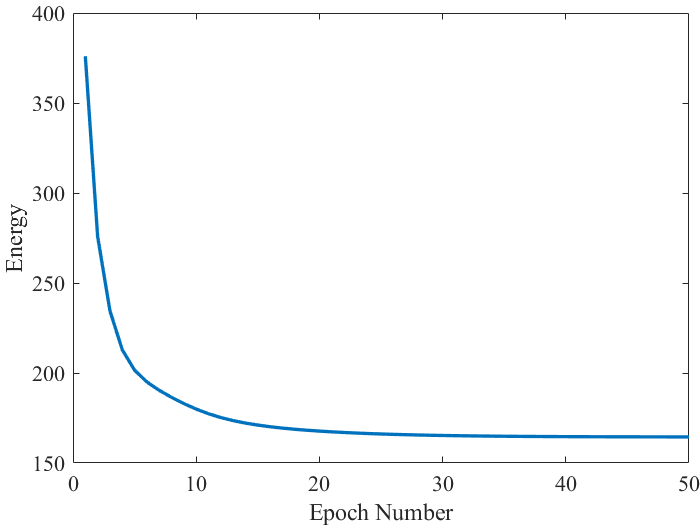}\\
  \begin{overpic}[width=.2\textwidth]{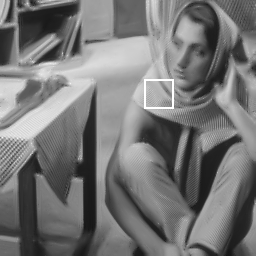}\put(50,0){\includegraphics[width=.1\textwidth]{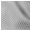}}\end{overpic}&
  \begin{overpic}[width=.2\textwidth]{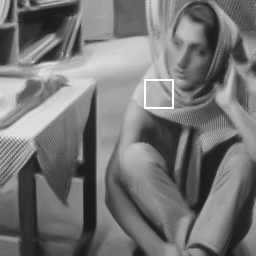}\put(50,0){\includegraphics[width=.1\textwidth]{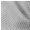}}\end{overpic}&
  \begin{overpic}[width=.2\textwidth]{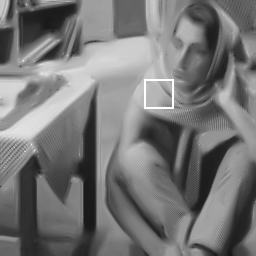}\put(50,0){\includegraphics[width=.1\textwidth]{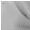}}\end{overpic}&
  \begin{overpic}[width=.2\textwidth]{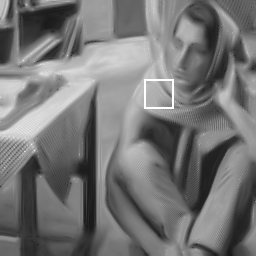}\put(50,0){\includegraphics[width=.1\textwidth]{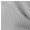}}\end{overpic}\\
  {\footnotesize\emph{s=15\%, $p=1$, 28.91dB}} &
  {\footnotesize\emph{s=15\%, $p=2$, 28.65dB}} &
  {\footnotesize\emph{s=5\%, $p=1$, 25.29dB}} &
  {\footnotesize\emph{s=5\%, $p=2$, 24.90dB}}
  \end{tabular}
    \caption{Restored results of Algorithm \ref{alg1} for $p=1,2$ and sampling rates $s=15\%, 5\%$.
    The first row shows the energy decay over iteration cycles.
    $p=1$: $\rho=1$, $r=1.5$;  $p=2$: $\rho=0.1$, $r=1.2$.
  }
    \label{fig:convergence}
  \end{figure}

\begin{table*}[htbp]
\centering
\scriptsize
\setlength{\tabcolsep}{9pt}
\renewcommand{\arraystretch}{0.9}
\caption{PSNR and SSIM comparisons of GL and HpL under different sampling rates.}
\label{table:PSNR}
\begin{tabular}{llcccccc}
\toprule
 &  & \multicolumn{3}{c}{PSNR} & \multicolumn{3}{c}{SSIM} \\
\cmidrule(lr){3-5} \cmidrule(lr){6-8}
Image & Method & 5\% & 10\% & 15\% & 5\% & 10\% & 15\% \\
\midrule
\multirow{3}{*}{Airplane}
& GL & 20.90 & 22.40 & 23.83 & 0.3125 & 0.4531 & 0.5447 \\
& sPDHG & 21.06 & 22.81 & 24.41 & 0.3614 & 0.5054 & 0.5973 \\
& Alg.~\ref{alg1} & 21.08 & 22.84 & 24.45 & 0.3744 & 0.5174 & 0.6056 \\
\midrule
\multirow{3}{*}{Barbara}
& GL & 22.40 & 24.60 & 26.13 & 0.4294 & 0.5718 & 0.6593 \\
& sPDHG & 22.91 & 25.23 & 27.00 & 0.4807 & 0.6190 & 0.7016 \\
& Alg.~\ref{alg1} & 22.98 & 25.28 & 27.03 & 0.4923 & 0.6264 & 0.7065 \\
\midrule
\multirow{3}{*}{Beans}
& GL & 22.84 & 24.49 & 26.22 & 0.4369 & 0.5359 & 0.6082 \\
& sPDHG & 23.59 & 25.20 & 27.36 & 0.4814 & 0.5759 & 0.6451 \\
& Alg.~\ref{alg1} & 23.75 & 25.29 & 27.46 & 0.4908 & 0.5791 & 0.6454 \\
\midrule
\multirow{3}{*}{Boat}
& GL & 21.80 & 23.23 & 24.54 & 0.3128 & 0.4422 & 0.5488 \\
& sPDHG & 21.99 & 23.49 & 25.01 & 0.3524 & 0.4891 & 0.5926 \\
& Alg.~\ref{alg1} & 22.01 & 23.50 & 25.08 & 0.3637 & 0.5007 & 0.6016 \\
\midrule
\multirow{3}{*}{Bridge}
& GL & 20.44 & 21.70 & 22.54 & 0.2794 & 0.4049 & 0.4907 \\
& sPDHG & 20.54 & 21.90 & 22.76 & 0.3103 & 0.4442 & 0.5230 \\
& Alg.~\ref{alg1} & 20.52 & 21.88 & 22.75 & 0.3210 & 0.4549 & 0.5327 \\
\midrule
\multirow{3}{*}{Cameraman}
& GL & 21.03 & 21.67 & 23.01 & 0.2424 & 0.3314 & 0.4248 \\
& sPDHG & 21.23 & 21.94 & 23.30 & 0.2760 & 0.3659 & 0.4546 \\
& Alg.~\ref{alg1} & 21.19 & 21.96 & 23.25 & 0.2822 & 0.3694 & 0.4541 \\
\midrule
\multirow{3}{*}{Clock}
& GL & 22.80 & 24.89 & 25.96 & 0.3237 & 0.4431 & 0.5226 \\
& sPDHG & 23.07 & 25.42 & 26.56 & 0.3547 & 0.4788 & 0.5567 \\
& Alg.~\ref{alg1} & 23.04 & 25.47 & 26.53 & 0.3604 & 0.4848 & 0.5622 \\
\midrule
\multirow{3}{*}{Couple}
& GL & 21.76 & 23.48 & 24.40 & 0.3007 & 0.4532 & 0.5549 \\
& sPDHG & 21.92 & 23.91 & 24.83 & 0.3473 & 0.5090 & 0.6031 \\
& Alg.~\ref{alg1} & 21.91 & 23.97 & 24.82 & 0.3588 & 0.5210 & 0.6109 \\
\midrule
\multirow{3}{*}{Hill}
& GL & 23.32 & 25.04 & 26.03 & 0.3183 & 0.4744 & 0.5559 \\
& sPDHG & 23.65 & 25.49 & 26.58 & 0.3644 & 0.5240 & 0.6017 \\
& Alg.~\ref{alg1} & 23.71 & 25.52 & 26.58 & 0.3778 & 0.5348 & 0.6087 \\
\midrule
\multirow{3}{*}{House}
& GL & 25.28 & 27.51 & 29.17 & 0.2642 & 0.3867 & 0.4674 \\
& sPDHG & 25.93 & 28.41 & 30.09 & 0.2978 & 0.4176 & 0.4937 \\
& Alg.~\ref{alg1} & 25.96 & 28.42 & 30.11 & 0.3028 & 0.4201 & 0.4965 \\
\midrule
\multirow{3}{*}{Man}
& GL & 22.21 & 24.04 & 24.95 & 0.3124 & 0.4469 & 0.5313 \\
& sPDHG & 22.49 & 24.55 & 25.43 & 0.3599 & 0.5031 & 0.5806 \\
& Alg.~\ref{alg1} & 22.52 & 24.60 & 25.46 & 0.3719 & 0.5129 & 0.5888 \\
\midrule
\multirow{3}{*}{Peppers}
& GL & 22.29 & 24.18 & 25.85 & 0.4723 & 0.5878 & 0.6742 \\
& sPDHG & 22.84 & 25.04 & 26.74 & 0.5279 & 0.6467 & 0.7219 \\
& Alg.~\ref{alg1} & 22.87 & 25.17 & 26.78 & 0.5373 & 0.6555 & 0.7260 \\
\midrule
\multirow{3}{*}{Average}
& GL & 22.25 & 23.93 & 25.22 & 0.3338 & 0.4610 & 0.5486 \\
& sPDHG & 22.60 & 24.45 & 25.84 & 0.3762 & 0.5066 & 0.5893 \\
& Alg.~\ref{alg1} & 22.63 & 24.49 & 25.86 & 0.3861 & 0.5148 & 0.5949 \\
\bottomrule
\end{tabular}
\end{table*}

We compare Algorithm \ref{alg1} with GL and sPDHG for image inpainting.
Table \ref{table:PSNR} reports the PSNR and SSIM results under different sampling rates.
HpL consistently outperforms GL across all test images and sampling rates, indicating that the hypergraph structure is more suitable than the graph structure for image interpolation problems.
The improvement becomes more pronounced as the sampling rate increases from 5\% to 15\%.
Since Algorithm \ref{alg1} and sPDHG are two solvers for HpL, they have similar performance.
Algorithm \ref{alg1} achieves slightly better SSIM than sPDHG.
The restored results for three test images are shown in Figure \ref{fig:inpainting}.
We observe that HpL suppresses spikes more effectively than GL, which explains its superior performance.

\begin{figure}[t]
  \centering
  \begin{tabular}{@{~}c@{~}c@{~}c@{~}c@{}}
  \begin{overpic}[width=.2\textwidth]{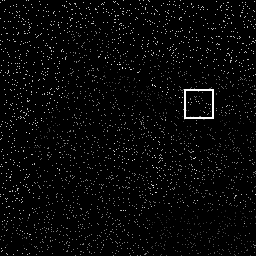}\put(50,0){\includegraphics[width=.1\textwidth]{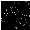}}\end{overpic}
  \begin{overpic}[width=.2\textwidth]{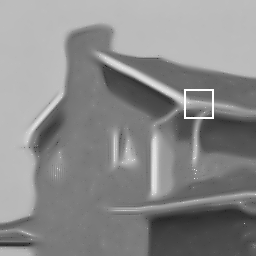}\put(50,0){\includegraphics[width=.1\textwidth]{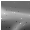}}\end{overpic}
  \begin{overpic}[width=.2\textwidth]{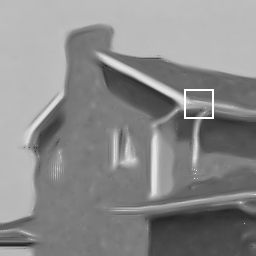}\put(50,0){\includegraphics[width=.1\textwidth]{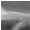}}\end{overpic}
  \begin{overpic}[width=.2\textwidth]{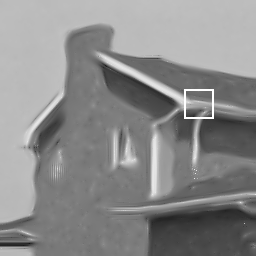}\put(50,0){\includegraphics[width=.1\textwidth]{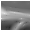}}\end{overpic}\\
  \begin{overpic}[width=.2\textwidth]{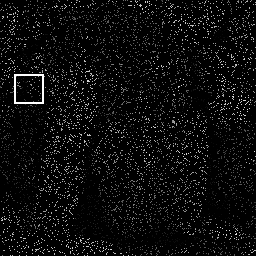}\put(50,0){\includegraphics[width=.1\textwidth]{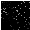}}\end{overpic}
  \begin{overpic}[width=.2\textwidth]{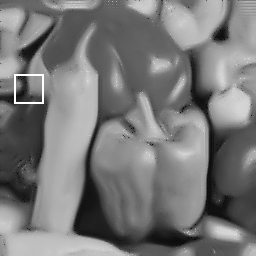}\put(50,0){\includegraphics[width=.1\textwidth]{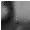}}\end{overpic}
  \begin{overpic}[width=.2\textwidth]{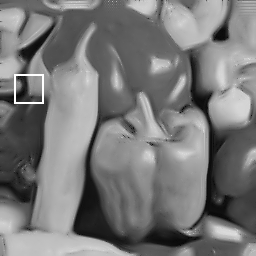}\put(50,0){\includegraphics[width=.1\textwidth]{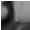}}\end{overpic}
  \begin{overpic}[width=.2\textwidth]{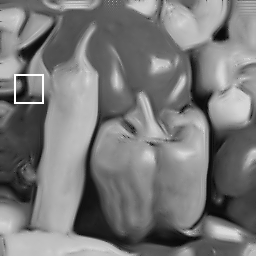}\put(50,0){\includegraphics[width=.1\textwidth]{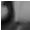}}\end{overpic}\\
  \begin{overpic}[width=.2\textwidth]{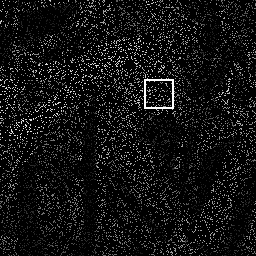}\put(50,0){\includegraphics[width=.1\textwidth]{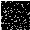}}\end{overpic}
  \begin{overpic}[width=.2\textwidth]{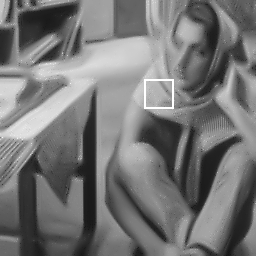}\put(50,0){\includegraphics[width=.1\textwidth]{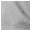}}\end{overpic}
  \begin{overpic}[width=.2\textwidth]{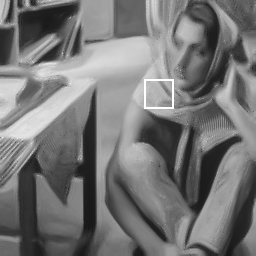}\put(50,0){\includegraphics[width=.1\textwidth]{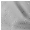}}\end{overpic}
  \begin{overpic}[width=.2\textwidth]{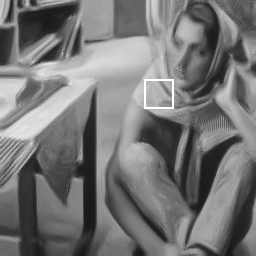}\put(50,0){\includegraphics[width=.1\textwidth]{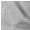}}\end{overpic}\\
  \end{tabular}
    \caption{Restored results of GL and HpL for three test images.
   From left to right: The observed pixels, GL, sPDHG, and Algorithm \ref{alg1}.
   From top to bottom (image/sampling rate): House/5\%, Peppers/10\%, Barbara/15\%. We set $\rho=0.05$ and $r=1.2$ in Algorithm \ref{alg1} for all image inpainting experiments.
  }
    \label{fig:inpainting}
  \end{figure}

\begin{table}[htbp]
\centering
\setlength{\tabcolsep}{10pt}
\renewcommand{\arraystretch}{1.1}
\caption{Average running time  (seconds) over different test images.}
\label{table:time_inpainting}
\begin{tabular}{lccc}
\toprule
 & 5\% & 10\% & 15\% \\
\midrule
GL & 26.1 & 28.1 & 29.3 \\
sPDHG & 55.5 & 74.0 & 91.3 \\
Alg.~\ref{alg1} & 29.7 & 29.6 & 29.6 \\
\bottomrule
\end{tabular}
\end{table}

For both Algorithm \ref{alg1} and sPDHG, we set the iteration number to $10 \times 256^2$, corresponding to 10 full iteration cycles.
Table \ref{table:time_inpainting} reports the average running times of the three algorithms.
GL and Algorithm \ref{alg1} exhibit similar computational costs. This is because we set $K = 15$ for GL and $K = 1$ for Algorithm \ref{alg1} in Algorithm \ref{alg3}. In GL, the computation of the weight function is the most time-consuming step.
Algorithm \ref{alg1} achieves a clear speedup compared with sPDHG, which highlights the main computational advantage of the proposed algorithm.

\begin{table}[H]
  \centering
  \caption{Datasets for semi-supervised learning.}
    \begin{tabular}{lllllll}
    \toprule
    Dataset & MNIST & Citeseer & Cora & Cora & DBLP & Pubmed\\
            &       & citation & authorship & citation & authorship & citation \\
    \midrule
    $|V|$ & 70000 & 3306 & 2078 & 2708 & 41302 & 19717 \\
    $|E_H|$ & 70000 & 1079 & 1072 & 1516 & 22365 & 7963 \\
    \#classes & 10 & 6   & 7    & 7    & 6     & 3 \\
    \bottomrule
    \end{tabular}%
  \label{table:dataset}%
\end{table}

\subsection{Semi-supervised learning}
In the remainder of this section, we evaluate the algorithms for semi-supervised learning on several real-world datasets summarized in Table \ref{table:dataset}.
MNIST consists of 70,000 grayscale images of handwritten digits, where each image has size $28 \times 28$ and represents a digit from 0 to 9 \cite{lecun1998gradient}. We treat each image as a vertex in a point cloud embedded in $\mathbb{R}^{28 \times 28}$ and construct the corresponding graph/hypergraph in the same manner as described in Section \ref{section:inpainting}.
Citeseer, Cora, DBLP, and Pubmed \cite{sen2008collective} are hypergraph datasets, where vertices represent documents and hyperedges follow from co-citation or co-authorship relationships. For these hypergraph datasets, we assign uniform weights $w_{i,j} = 1$.

As in the previous section, we compare Algorithm \ref{alg1} with GL and sPDHG on MNIST.
Note that when hyperedges have large cardinality, sPDHG may run out of memory.
We instead adopt SGD for hypergraph datasets.
Choosing an appropriate step size and stopping criterion for SGD is nontrivial. We use an empirically selected step size as suggested by the authors \cite{zhang2017re} and set the iteration number to 500.
Both Algorithm \ref{alg1} and sPDHG are convergent.
We terminate the algorithms when the relative error at the $(c+1)$-th cycle satisfies
\begin{equation*}
  \frac{\|u^{c+1}-u^c\|^2}{\|u^{c+1}\|^2}\leq \varepsilon.
\end{equation*}
We set $\varepsilon=5\times 10^{-4}$ for MNIST and $\varepsilon=5\times 10^{-5}$ for the hypergraph datasets.
The clique extension of the hypergraph datasets is used for GL.

Table \ref{table:SSL} presents the test accuracy of different algorithms. The results are obtained by averaging over 10 independent runs.
From the results on MNIST, we observe that GL fails to produce meaningful results under extremely low labeling rates. In contrast, HpL does not suffer from this issue. Moreover, Algorithm \ref{alg1} achieves significantly higher prediction accuracy than both sPDHG and GL.
For the hypergraph datasets, the performance gap among different algorithms is less pronounced. Nevertheless, Algorithm \ref{alg1} still attains the highest mean accuracy.

\begin{table}[htbp]
\centering
\caption{Accuracy (\%) of semi-supervised learning on test datasets.}
\label{table:SSL}
\begin{tabular}{cccccccccc}
\toprule
 &  & MNIST &  &  & Citeseer & Cora & Cora & DBLP & Pubmed \\
 &  &  &  &  & citation & author & citation & author & citation \\
\midrule
\multirow{3}{*}{0.05\%}
& GL & 17.8
& \multirow{3}{*}{5\%}
& GL & 58.6 & 53.9 & 68.4 & 74.9 & 72.4 \\
& sPDHG & 65.4
& & SGD & 58.9 & 51.5 & 69.5 & 75.1 & 79.2 \\
& Alg.~\ref{alg1} & 75.8
& & Alg.~\ref{alg1} & 58.9 & 53.8 & 71.3 & 75.8 & 78.5 \\
\midrule
\multirow{3}{*}{0.5\%}
& GL & 86.7
& \multirow{3}{*}{10\%}
& GL & 64.8 & 64.0 & 77.7 & 79.4 & 78.8 \\
& sPDHG & 89.8
& & SGD & 64.9 & 61.5 & 77.5 & 79.0 & 82.1 \\
& Alg.~\ref{alg1} & 93.4
& & Alg.~\ref{alg1} & 64.6 & 64.2 & 77.9 & 79.9 & 82.1 \\
\bottomrule
\end{tabular}
\end{table}

\begin{table}[htbp]
\centering
\caption{Average running time (seconds) for semi-supervised learning.}
\label{table:time_SSL}
\begin{tabular}{c c c c c c cc}
\toprule
 & MNIST &  & Citeseer & Cora & Cora & DBLP & Pubmed \\
 &  &  & citation & author & citation & author & citation \\
\midrule
GL & 31.0 & GL & 0.02 & 0.03 & 0.02 & 0.84 & 0.14 \\
sPDHG & 742.5 & SGD & 0.5 & 0.8 & 0.4 & 86.9 & 13.0 \\
Alg.\ \ref{alg1} & 122.3 & Alg.\ \ref{alg1} & 0.9 & 1.1 & 1.6 & 21.3 & 4.9 \\
\bottomrule
\end{tabular}
\end{table}

The average running time of different algorithms are shown in Table \ref{table:time_SSL}.
On Citeseer, Cora, and PubMed, the number of hyperedges is relatively small. All three algorithms can be executed efficiently.
For the two largest datasets, MNIST and DBLP, we observe that the proposed algorithm is several times faster than both sPDHG and SGD.
Note that for hypergraph datasets, we use uniform weights $w_{i,j} = 1$, which avoids solving a linear system in the $u$-subproblem.

\section{Conclusion}
In this paper, we proposed an operator-splitting method for the non-smooth hypergraph $p$-Laplacian.
It processed hyperedges sequentially in a Gauss–Seidel fashion, thereby simplifying the computational cost.
The convergence of the proposed algorithm was established.
Numerical experiments on image sparse inpainting and semi-supervised learning demonstrated that the proposed method exhibited superior computational efficiency compared with the stochastic primal–dual algorithm and the subgradient descent method, particularly on large-scale hypergraph datasets.

\section*{Acknowledgments}
The authors acknowledge support from DESY (Hamburg, Germany), a member of the Helmholtz Association HGF.









%
%

\bibliographystyle{unsrt}
\bibliography{references}

\end{document}